\documentclass[oneside,11pt]{amsart}
\usepackage{amsmath,amssymb,amsthm, mathrsfs, mathtools}
\usepackage{amscd,mathtools}
\usepackage{enumitem} \setlist{nosep} % Change enumeration labelling
\usepackage[utf8]{inputenc}
\usepackage[english]{babel}
\usepackage{color}
\usepackage[pagebackref,breaklinks,hidelinks,unicode]{hyperref}
\usepackage{float}
\usepackage{comment}
\usepackage{mathtools}
\usepackage{csquotes}
\usepackage[font=small,justification=centering]{caption}
\usepackage{subcaption}
\usepackage[dvipsnames]{xcolor}
\usepackage{mathabx} % Includes \sqsubset etc.
\usepackage{xfrac} % Nice quotients
\usepackage[normalem]{ulem} % Contains \sout

\usepackage[margin=3cm]{geometry}

\usepackage{tikz}
\usepackage{tikz-cd}
\usetikzlibrary{automata}
\usetikzlibrary{calc,decorations.pathreplacing}
\usetikzlibrary{arrows}
\usetikzlibrary{decorations.pathreplacing}
\usetikzlibrary{intersections}

\makeatletter
    \newtheorem*{rep@theorem}{\rep@title}
    \newcommand{\newreptheorem}[2]{%
    \newenvironment{rep#1}[1]{%
    \def\rep@title{#2 \ref{##1}}%
    \begin{rep@theorem}}%
    {\end{rep@theorem}}}
\makeatother

\newtheorem{theorem}{Theorem}[section]
\newtheorem{proposition}[theorem]{Proposition}
\newreptheorem{proposition}{Proposition}
\newtheorem{corollary}[theorem]{Corollary}
\newtheorem{lemma}[theorem]{Lemma}

\newtheorem{Set up}[theorem]{Set-up}
\newtheorem{question}{Question}
\newtheorem{introthm}{Theorem}
\newtheorem{introcor}[introthm]{Corollary}

\theoremstyle{definition}

\newtheorem{definition}[theorem]{Definition}

\newtheorem*{answer*}{Answer}
\newtheorem*{application*}{Application}

\DeclarePairedDelimiterX{\Norm}[1]{\lVert}{\rVert}{#1}
\theoremstyle{definition}

\renewcommand*{\backrefalt}[4]{\ifcase #1 (Not cited).\or (Cited p.~#2).\else (Cited pp.~#2).\fi} % Cited on...

\newcounter{shcount}
 
\newcounter{enumlabelcount}
\makeatletter\newcommand\enumlabel[1][]{\item[#1]
    \refstepcounter{enumlabelcount}\def\@currentlabel{#1}}\makeatother

\newcounter{claimcount}

\newenvironment{claim*}[1]{\par\vspace{2mm}\noindent
    \underline{Claim:}\hspace{2mm}#1}{}

\title{Morse subsets of injective spaces are strongly contracting}
\author{Alessandro Sisto}
	\address{Maxwell Institute and Department of Mathematics, Heriot-Watt University, Edinburgh, UK}
	\email{a.sisto@hw.ac.uk}
	
		\address{Department of mathematics, University of Toronto, Toronto, Ontario, Canada M5S 2E4}
	\email{abdul.zalloum@utoronto.ca}

\author{Abdul Zalloum}

\begin{document}
\maketitle

\begin{abstract} 
We show that a quasi-geodesic in an injective metric space is Morse if and only if it is strongly contracting. Since mapping class groups and, more generally, hierarchically hyperbolic groups act properly and coboundedly on injective metric spaces, we deduce various consequences relating, for example, to growth tightness and genericity of pseudo-Anosovs/Morse elements. Moreover, we show that injective metric spaces have the Morse local-to-global property and that a non-virtually-cyclic group acting properly and coboundedly on an injective metric space is acylindrically hyperbolic if and only if contains a Morse ray. We also show that strongly contracting geodesics of a space stay strongly contracting in the injective hull of that space.
\end{abstract}

\section{Introduction} 
\subsection*{Hyperbolic-like directions} There are various notions that aim to capture the notion of ``hyperbolic-like" directions in a space or group (for sampling, see \cite{charneysultan:contracting,ACGH,QRT20,Russell2018ConvexityIH,DZ22,Tran2019}) meaning geodesics that behave like geodesics in a hyperbolic space. The weakest such notion is that of a \emph{Morse} geodesic, while the strongest of these notions is that of a \emph{strongly contracting} geodesic. The advantage of the former is that it is quasi-isometry invariant, while the latter has much stronger consequences, for example relating to acylindrical hyperbolicity \cite{Bestvina2014,EikeZalloum}, growth-tightness \cite{Arzhantseva2015}, growth rates of subgroups \cite{Xabi}, and genericity of strongly contracting elements \cite{Yang2018}; see the discussion below. Unfortunately, in general metric spaces these two notions are not equivalent, and in fact they are not even equivalent for Cayley graphs of mapping class groups \cite{Rafi2021}. It is therefore desirable to identify classes of metric spaces where the notions are equivalent, as is the case for CAT(0) spaces \cite{charneysultan:contracting,Cashen:Morse_CAT0} and 1-skeleta of CAT(0) cube complexes \cite{Genevois2020}.
%, and it has recently been shown that it is the case for Garside groups as well \cite{Garsidestrong}.

\subsection*{Injective spaces} A metric space is said to be \emph{injective} if any collection of pairwise intersecting balls have a total intersection. A related notion is that of \emph{Helly graph}, a graph where this property is required to hold for combinatorial balls. Injectivity has drawn increasing interest in recent years and it is meant to capture some form of nonpositive curvature. Indeed, for example, these spaces admit bicombings with various desirable properties \cite{LANG2013,Descombes2014}. Moreover, \emph{Helly groups}, that is, groups acting properly and cocompactly on Helly graphs, have many remarkable properties such as biautomaticity \cite{CCGHO20}. These groups form a large class including cubulated groups, finitely presented graphical C(4)-T(4) small cancellation groups \cite{CCGHO20}, weak Garside groups of finite type \cite{Huang2021}. Injectivity has also played a crucial role in the proof of semi-hyperbolicity of mapping class groups \cite{haettelhodapetyt:coarse} (see also \cite{DMS20}). Namely, Haettel-Hoda-Petyt showed that hierarchically hyperbolic groups, which include mapping class groups (see \cite{HHS1} and references therein), admit proper cobounded actions on injective metric spaces.
%; their work is of particular interest to the present article.
%Many more groups of interest admit proper cobounded actions on injective spaces such as 
Our main theorem is the following.

% Well-behaved actions on hyperbolic spaces are known to have some very strong algebraic, dynamical, and combinatorial consequences regarding the structure of the underlying group \cite{gromov:hyperbolic}, \cite{masurminsky:geometry:1}, \cite{osin:acylindrically}, \cite{behrstockhagensisto:hierarchically:1}. Since geoetric group theory concerns understanding groups using their actions on metric spaces, a main theme in the area has been to construct actions of groups on hyperbolic spaces with the goal of understanding the structure of such groups \cite{hatchervogtmann:complex}, \cite{bestvinafeighn:hyperbolicity}, \cite{bestvinabrombergfujiwarasisto:acylindrical}, \cite{genevois:hyperbolicities}, \cite{hagen:weak}, \cite{PSZCAT}.

% The presence of some ``hyperbolic-like" geodesics in a group $G$ is usually an indication that $G$ admits a non-trivial action on some hyperbolic space. The strongest of such ``hyperbolic-like" geodesics  is that of a \emph{strongly contracting} geodesic while the weakest is that of a \emph{Morse} geodesic. The disparity between two such notions is vast: while the presence of a Morse geodesic in an arbitrary finitely generated group $G$ is not known to have any algebraic consequences \cite{Fink} regarding $G$, the presence of a strongly contracting geodesic in a group $G$, or in a space admitting a proper cobounded action by $G$, informs a great deal of the algebraic structure of the group $G$, for instance, it implies \emph{acylindrical hyperbolicity} of $G$ \cite{Bestvina2014}, \cite{EikeZalloum}. We show the following.

\begin{introthm}\label{thm:main}
A subset $Y$ of an injective metric space $X$ is Morse if and only if it is strongly contracting, quantitatively.
\end{introthm}

By ``quantitatively" we mean that the contraction constant and the Morse gauge determine each other. We note that since Helly graphs are coarsely dense in their injective hulls (see \cite[Proposition 3.12]{CCGHO20}), the theorem also holds for those.

A recent result related to Theorem \ref{thm:main} is that Morse elements are strongly contracting in standard Cayley graphs of Garside groups of finite type modded out by their centres \cite{Garsidestrong}.

We note that Morseness for $(9,0)$--quasi-geodesics actually suffices for strong contraction, see Corollary \ref{cor:9_0}. Interestingly, this is better than the case of $CAT(0)$ spaces, see the discussion below the aformentioned corollary.

\subsection*{Injective hulls} Before discussing consequences of our main theorem, we state another result which completes the picture of what happens when passing to the injective hull (but will not be discussed further below). We recall that any metric space $X$ embeds isometrically into a canonical injective metric space $E(X)$ called its \emph{injective hull}, which should be thought of as the smallest injective space containing $X$. Theorem \ref{thm:main} implies that in order for a geodesic ray $\gamma$ in $X$ to be Morse in $E(X)$, the geodesic $\gamma$ needs to be strongly contracting in $X$. Conversely, we show that if $\gamma$ is indeed strongly contracting in $X$, then it stays strongly contracting in $E(X)$:

\begin{introthm}
\label{thm:contracting_survives_intro}
Let $X$ be a geodesic metric space with injective hull $e:E\to E(X)$. Let $\gamma\subseteq X$ be a strongly contracting geodesic. Then $e(\gamma)$ is strongly contracting in $E(X)$, quantitatively.
\end{introthm}

In other words, a Morse geodesic in $X$ stays Morse in $E(X)$ if and only if $\gamma$ is strongly contracting in $X.$ We note that the same statement also holds for strongly contracting subsets satisfying Gromov's 4-point condition, see Theorem \ref{thm:contracting_survives}. We note that this stronger version implies the well-known fact that hyperbolic spaces are coarsely dense in their injective hulls, see Corollary \ref{cor:Lang}, which was first proven in \cite{LANG2013}.

At this point a natural question arises, which is what happens to rays of $X$ that are Morse but not strongly contracting when passing to the injective hull $E(X)$. We do not have a good guess for this, so we formulate the question vaguely:

\begin{question}
Let $\gamma\subseteq X$ be a Morse geodesic and let $e:X\to E(X)$ be the injective hull of $X$. Does $e(\gamma)$ retain any (weak) form of ``hyperbolicity"?
\end{question}

\subsection*{Consequences for mapping class groups}

We now point out various consequences of our main theorem when combined with known results from the literature, in the context of mapping class groups. Analogous results also hold for all hierarchically hyperbolic groups. We emphasise that contraction properties weaker than strong contraction do not suffice for any of the applications below (quasi-axes of pseudo-Anosov elements in Cayley graphs are known to satisfy a weaker form of contraction \cite{Behrstock_pA_are_Morse}).

First of all, since mapping class groups act properly and coboundedly on injective metric spaces, we obtain the first known geometric models for mapping class groups where pseudo-Anosov elements have strongly contracting axes (where by geometric model we mean a geodesic space being acted on properly and coboundedly, hence equivariantly quasi-isometric to the given group).

\begin{introthm}{\rm (Theorem \ref{thm:main} plus \cite{haettelhodapetyt:coarse})}
\label{morseiscontracting} Each mapping class group $G$ admits a proper cobounded action on a metric space $X$ such that the following hold:

\begin{enumerate}
\item Morse and strongly contracting sets are equivalent in $X,$ and
\item pseudo-Anosov elements of $G$ all have strongly contracting quasi-axes in $X.$
\end{enumerate}
\end{introthm}

% It is worth noting that the space $X$ in Corollary \ref{morseiscontracting} is the first known metric space acted on geomtrically by mapping class groups where the conclusions in items 1 and 2 hold.
This theorem further highlights the significance of \cite{haettelhodapetyt:coarse} and encourages further study of injective metric spaces in light of their recent influence on geometric group theory. On this note, we point out that although this may not be apparent from their construction, the ideas behind recent work of \cite{PSZCAT} constructing hyperplanes and hyperbolic models for CAT(0) spaces originate from the study of injective metric spaces and their $d^{\infty}$-like distance.

We remark that it was previously known that mapping class groups are quasi-isometric to spaces where Morse subsets are strongly contracting, namely, CAT(0) cube complexes \cite{Petyt21}. However, such a quasi-isometry is not equivariant as mapping class groups do not act on these spaces.

\subsubsection*{On genericity of pseudo-Anosovs}

A long standing conjecture due to Thurston \cite{FarbProblems} states that pseudo-Anosov elements are generic with respect to the counting measure in each Cayley graph of a mapping class group $G$.  Work of Yang \cite{Yang2018} shows genericity of strongly contracting elements under general conditions that apply to proper cobounded actions with a strongly contracting element. Combining our main theorem with \cite{haettelhodapetyt:coarse} and \cite{Yang2018} we get:

% A striking theorem due to Yang \cite{Yang2018} states that if each Cayley graph of a mapping class group contains a strongly contracting geodesic, then Thurston's conjecture follows.

% However, a statement of such generality is hopeless: recent work of Rafi and Yvon \cite{Rafi2021} exhibits a finite generating set for the mapping class group of a 5-punctured sphere where the notions of Morse and strongly contracting are not equivalent. Hence, items 1 and 2 in the following corollary are in a sense as good as one could hope for regarding axis of pseudo-Anosov elements.

% In regards to genericity of pseudo-Anosov elements, and passing through work of \cite{Yang2018}, we get the following.

\begin{introthm}{\rm (Theorem \ref{thm:main} plus \cite{haettelhodapetyt:coarse} plus \cite{Yang2018})}
\label{cor:genericity} Let $G$ be a mapping class group and let $X$ be an injective space acted on properly and coboundedly by $G.$ The collection of pseudo-Anosov elements is generic with respect to the counting measures in balls of $X$. 
\end{introthm}

Here, ``genericity with respect to the counting measure in balls" is defined in a natural way in terms of counting orbit points in balls of increasing radius, see \cite[Subsection 1.1]{Yang2018}.

Although the space $X$ from Theorem \ref{morseiscontracting} is the first geometric model of the mapping class group where the Morse and strongly contracting notions are equivalent, it is not the first space where pseudo-Anosov elements are generic with respect to the counting measure, as Choi \cite{Choi2021PseudoAnosovsAE} showed that there are finite generating sets of any mapping class group where pseudo-Anosov elements are generic with respect to the counting measure in that Cayley graph. 
%Unfortunately, neither of our works provides a full resolution to Thurston's conjecture, but can be viewed as strong evidence towards its validity.

\subsubsection*{Growth-tightness} Another consequence of the presence of strongly contracting geodesics relates to \emph{growth tightness} of the action. For word metrics, growth tightness is a condition comparing growth rates of a given group with the growth rates of its quotient, and it was originally introduced by de la Harpe and Grigorchuk as a scheme to show that a given group is Hopfian \cite{growth_tight}.  Work of Arzhantseva, Cashen, and Tao \cite{Arzhantseva2015} generalises this notion to group actions, and they show that cobounded actions that admit a strongly contracting element are growth tight. In \cite[Question 3]{Arzhantseva2015}, the authors ask whether the action of a mapping class group $G$ on its Cayley graph(s) is growth tight. Combining Theorem \ref{thm:main} with \cite{haettelhodapetyt:coarse} and \cite{Arzhantseva2015}, we obtain the following.

\begin{introthm}\label{cor:growthtight}{\rm (Theorem \ref{thm:main} plus \cite{haettelhodapetyt:coarse} plus \cite{Arzhantseva2015})}  Each mapping class group admits a growth-tight proper cobounded action on a metric space. 
\end{introthm}

Although Theorem \ref{cor:growthtight} does not exactly answer their question, it does provide the first geometric model for the mapping class group where the action is growth-tight (note that while the action of a mapping class group on the corresponding Teichm\"uller space with the Teichm\"uller metric is growth-tight \cite{Arzhantseva2015}, such an action is not cobounded).

\subsection*{Consequences for groups acting on injective spaces} We now discuss consequences of our main theorem for Helly groups and other groups acting on injective metric spaces. We start with another large-scale geometric property of injective metric spaces that we obtain, namely the \emph{Morse local-to-global} property.

\subsubsection*{Morse local-to-global}  
The Morse local-to-global property was introduced in \cite{Morse-local-to-global} with the goal of generalising results and arguments about hyperbolic groups to more general groups. Such spaces and groups have been proven to enjoy an abundance of desirable properties such as the existence of an automatic structure for Morse geodesics \cite{Cordes2022,Huges2022}, combination type theorems for stable subgroups \cite{Morse-local-to-global}, rationality of the growth of stable subgroups \cite{Cordes2022}, and some implications regarding the action of such groups on their Morse boundaries \cite{Cordes2022} (see \cite{charneysultan:contracting,cordes:morse} for the notion of Morse boundary and \cite{durhamtaylor:convex, tran:onstrongly} for the definition of a stable subgroup). Our second main theorem shows that any metric space where the notions of Morse and strongly contracting are equivalent must have the Morse local-to-global property.

%\cite{charneysultan:contracting,cordes:morse}, \cite{murray:topology}, \cite{Liu2021}, \cite{Cashen2019},

\begin{introthm}\label{thm:injectiveareMLTG} If $X$ is a metric space where Morse geodesics are strongly contracting quantitatively, then $X$ has the Morse local-global-property. In particular, injective metric spaces have the Morse local-global-property. 
\end{introthm}

In \cite{Morse-local-to-global}, the authors prove that CAT(0) spaces and hierarchically hyperbolic spaces with the \emph{bounded domain dichotomy} (a minor condition that it met by all the main examples) have the Morse local-to-global property. Theorem \ref{thm:injectiveareMLTG} recovers the Morse local-to-global property for CAT(0) spaces and when combined with \cite{haettelhodapetyt:coarse}, it establishes the Morse local-to-global property for all hierarchically hyperbolic spaces.

\begin{introcor} \rm{(Theorem \ref{thm:injectiveareMLTG} plus \cite{haettelhodapetyt:coarse})}  Every hierarchically hyperbolic space has the Morse local-to-global property.
\end{introcor}

%\subsection{Consequences for groups acting on injective spaces}
\subsubsection*{Acylindrical hyperbolicity}

Strong contraction often is the key to using the machinery of projection complexes from \cite{Bestvina2014} (see also \cite{bestvinabrombergfujiwarasisto:acylindrical}). Indeed, in Subsection \ref{subsec:a.h} we will use the Morse local-to-global property to obtain a strongly contracting element, and then applying the aforementioned machinery we prove:

\begin{introthm}
\label{thm:ah_intro}
Let $G$ be a group acting properly and coboundedly on an injective space $X$, and suppose that $G$ is not virtually cyclic. Then $G$ is acylindrically hyperbolic if and only if it contains a Morse ray.
\end{introthm}

One reason of interest in this kind of statements lies in the fact that, while it is not known whether being acylindrically hyperbolic is a quasi-isometry invariant, containing a Morse ray is. Therefore, it is useful to know under what circumstances containing a Morse ray is equivalent to being acylindrically hyperbolic. In particular, we obtain the following corollary.

\begin{introcor}
If $G,H$ are quasi-isometric Helly groups, then $G$ is acylindrically hyperbolic if and only if $H$ is.
\end{introcor}

This corollary is worth comparing to the analogous statement for groups acting on CAT(0) cube complexes from \cite{genevois:hyperbolicities}.

\subsubsection*{Growth rates} 

Our final two consequences are about \emph{growth rates} of groups acting on injective metric spaces and their subgroups. Given an action of a group $G$ on a metric space and a subset $A$ of $G$, the growth rate in $X$ of $A$ (if well-defined) measures, roughly, the exponential growth rate of the orbits of $A$ in $X$. Recent work of Legaspi \cite{Xabi} provides various interesting results regarding growth rates of groups admitting a proper action on a metric space $X$ with a strongly contracting element. Combining his work with Theorem \ref{thm:main}, we get the following.

\begin{introthm}{\rm (Theorem \ref{thm:main} plus \cite{Xabi})}\label{thm:growthrate}
Let $G$ be a group admitting a proper cobounded action on an injective metric space $X$. Assume further that $G$ is not virtually cyclic and that it contains a Morse ray. If $H<G$ is an infinite-index quasi-convex subgroup of $G$, then the growth rate in $X$ of $H$ is strictly smaller than the growth rate in $X$ of $G$.
% If $\lambda_{G,X},\lambda_{H,X}, \lambda_{{G/H,X}}$ denote the growth rates in $X$ of the respective sets, then we have:

% \begin{enumerate}
%     \item $\lambda_H< \lambda_G$, and
%     \item $\lambda_G=\lambda_{G/H}$.
% \end{enumerate}

% In particular, the statement applies to all HHGs and Helly groups.
\end{introthm}

In the context of groups acting properly and coboundedly on injective spaces, the above theorem provides an analogue of \cite[Theorem A]{Cordes2022} which states that the growth rate of any stable subgroup of a virtually torsion-free Morse local-to-global group is strictly smaller than that of the ambient group. An immediate consequence of Theorem \ref{thm:growthrate} is the following.

\begin{introcor} Let $X$ be an injective metric space acted on properly and coboundedly by a mapping class group $G$ and let $H<G$ be a convex cocompact subgroup. Then the growth rate in $X$ of $H$ is strictly smaller than the growth rate in $X$ of $G$. 
\end{introcor}

%We remark that since injective metric spaces are Morse local-to-global (Theorem \ref{thm:injectiveareMLTG}), 
The above corollary is a companion result to \cite[Corollary C]{Cordes2022} which shows the same statement for the orbit of a convex cocompact subgroup in any Cayley graph of its ambient mapping class group. It also relates to Gekhtman's work who showed that the same result holds for the action of a mapping class group on its Teichmüller space  \cite{Gekhtmen_convex_cocompact}.
\subsection*{Outline}

After some general preliminaries in Section \ref{sec:prelim}, we prove Theorem \ref{thm:main}(=Theorem \ref{thm:main_body}) in Section \ref{sec:proof_main}. In Section \ref{sec:MLTG} we prove the Morse local-to-global property, Theorem \ref{thm:injectiveareMLTG}(=Proposition \ref{thm:MLTG_body}), which we then use in the proof of Theorem \ref{thm:ah_intro}(=Corollary \ref{thm:ah_body}). Finally, in Section \ref{sec:contr_survives} we prove a more general version of Theorem \ref{thm:contracting_survives_intro}, namely Theorem \ref{thm:contracting_survives}.

\subsection*{Acknowledgements} The authors are thankful to Carolyn Aboott, Jason Behrstock, Ilya Gekhtman, Anthony Genevois, Thomas Haettel, Xabi Legaspi, Harry Petyt, Kasra Rafi, Mireille Soergel, and Bert Wiest for helpful comments, questions, and discussions. The authors would also like to thank an anonymous referee for useful comments and for being exceptionally quick.

\section{Preliminaries}
\label{sec:prelim}

% \begin{definition}
% An isometric action of a group $G$ on a metric space $X$ is said to be \emph{cobounded} if there exist a point $p \in X$ and a constant $C$ such that for any $x \in X,$ we have $d(x, g \cdot p) \leq C$ for some $g \in G.$ It is said to be \emph{proper} if for any ball $B \subset X$ we have $|\{g| B \cap g \cdot B \neq \emptyset\}|< \infty$.
% %An action is said to be \emph{geometric} if it is both proper and cobounded.
% \end{definition}

%Note that some authors define geometric actions as those that are proper and cocompact as opposed to proper and cobounded.
We will use the notation $B(x,r)$ to denote the (closed) ball of radius $r$ centered at $x.$
\begin{definition} (Injective spaces)
A metric space $X$ is said to be \emph{injective} if for any $\{x_i\}$ in $X$ and $\{r_i\} \in \mathbb{R}^+$, we have 

$$d(x_i,x_j) \leq r_i+r_j \text{ for all } i\neq j \implies \underset{i}{\cap}B(x_i,r_i) \neq \emptyset.$$
\end{definition}

\begin{definition}(Geodesic bicombing) A \emph{geodesic bicombing} on a metric space $(X,d)$ is a map $\gamma:X \times X \times[0,1] \rightarrow X$ such that for any two distinct points $x,y \in X$ the map $[0, d(x,y)] \rightarrow X$ given by $$t \mapsto \gamma_{x,y}\left(\frac{t}{d(x,y)}\right)$$ is a (unit-speed) geodesic, where $\gamma_{x,y}$ abbreviates the map $\gamma(x,y,-)$. A geodesic bicombing $\gamma$ is said to be \emph{conical} for any $x,y, x',y' \in X$ and $t \in [0,1]$, we have

$$d(\gamma_{x,y}(t),\gamma_{x',y'}(t)) \leq (1-t)d(\gamma_{x,y}(0), \gamma_{x',y'}(0))+td(\gamma_{x,y}(1), \gamma_{x',y}(1)).$$ It is said to be \emph{reversible} if for any $x,y \in X$, and $t \in [0,1]$, we have $\gamma_{x,y}(t)=\gamma_{y,x}(1-t).$

\end{definition}

\begin{proposition}\cite[Proposition 3.8]{LANG2013}\label{prop:bicombing} Each injective metric space $X$ admits a reversible conical geodesic bicombing which is $Isom(X)$-invariant. In particular, every injective metric space is geodesic.
\end{proposition}

\begin{definition}(Projections)\label{def: projections} Given a subset $Y$ of a metric space $X$ and $x\in X$, we define
$$\pi_{Y}(x)=\{y\in Y: d(x,y)\leq d(x,Y)+1\}\subseteq Y.$$ The set $\pi_Y(x)$ is called the \emph{coarse closest point projection} of $x$ to $Y.$ For a subset $B \subset X,$ we define $\pi_Y(B):=\underset{x \in B}{\cup} \pi_Y(x)$.

\end{definition}

\begin{definition}(Strongly contracting)
A subset $Y$ of a metric space $X$ is said to be \emph{$D$-strongly contracting} if for each ball $B$ disjoint from $Y$, we have $\text{diam}(\pi_Y(B)) \leq D$. It is said to be \emph{strongly contracting} if it is $D$-strongly contracting for some $ D.$

\end{definition}

We start with the following observation:

\begin{lemma}\label{rmk:proj_geod}
    Given a subset $Y$ of a metric space $X$, $x\in X$, and a point $p$ along a geodesic from $x$ to some $y\in \pi_Y(x)$, we have $d(p,y)\leq d(p,Y)+1$.
\end{lemma}

\begin{proof}
This follows from the following chain of inequalities which holds for any $z\in Y$, and taking the an infimum over all $z$:
$$d(x,Y)\leq (d(x,y)-d(p,y))+d(p,z)\leq d(x,Y)+1 -d(p,y) + d(p,z).$$
\end{proof}

\begin{definition}(Morse) Given a map $M: \mathbb{R}^+ \times \mathbb{R}^+ \rightarrow \mathbb{R}^+$, a subset $Y$ of a metric space is said to be \emph{$M$-Morse} if every $(\lambda,\epsilon)$-quasi-geodesic $\beta$ with endpoints on $Y$ remains in the $M(\lambda, \epsilon)$-neighborhood of $Y$. A subset $Y$ is said to be Morse if there exists some map $M$ such that $Y$ is $M$-Morse. In this case, we call $M$ a \emph{Morse gauge} for $Y$.
\end{definition}

We recall \cite[Lemma 2.5]{QRT19}  which roughly says that given a quasi-geodesic $\beta$ and a point $p$, getting from $p$ to the closest point in $\beta$ and then moving along $\beta$ describes a quasi-geodesic. Although \cite[Lemma 2.5]{QRT19} is stated for CAT(0) spaces, such an assumption is not used in their proof (alternatively, see the proofs of \cite[Lemma 2.2]{cordes:morse} and \cite[Proposition 4.2]{ACGH} where the authors prove similar statements without assuming the underlying metric space is CAT(0)).

\begin{lemma}\cite[Lemma 2.5]{QRT19}
\label{lem:QRT}
Let $\beta$ be a $(q,Q)$-quasi-geodesic and let $x \in X.$ Let $y \in \beta$ satisfy $d(x,y)=d(x,\beta)$ and let $z \in \beta$. Then the concatenation $[x,y] \cup \beta_{yz}$ is a $(3q,Q)$-quasi-geodesic, where $[x,y]$ is any geodesic connecting $x,y$ and $\beta_{yz}$ is the subsegment of $\beta$ connecting $y,z.$
\end{lemma}

\section{Morse is the same as strongly contracting in injective metric spaces}\label{sec:proof_main}

For a path $\beta$, if $x,y \in \beta,$ we use $\beta_{yz}$ to denote the subsegment of $\beta$ connecting $y,z.$ The following lemma, which is well-known, states that for any three points $x_1,x_2,x_3$ in an injective metric space, there exists an isometrically embedded tripod with the appropriate side lengths. More precisely, we have the following.

\begin{lemma}\label{lem:tripod} (Existence of tripods)
Let $X$ be an injective metric space. For any $x_1,x_2,x_3 \in X$ there exists a point $p\in X$ and geodesics $\beta^i$ from $x_i$ to $p$ such that for all $i\neq j$ the concatenation of $\beta_i$ and the reverse of $\beta_j$ is a geodesic from $x_i$ to $x_j$.
% $\beta^{ij}$ connecting $x_i,x_j$ and a point $p \in \underset{i \neq j}{\cap} \beta^{ij}$ such that $(\beta^{ij})_{x_ip} \cup (\beta^{ij})_{px_j}$ is a geodesic for all $i \neq j.$
\end{lemma}

\begin{proof}
Let $x_i$ be as in the statement and let $a_1,a_2, a_3$ be non-negative integers such that $d(x_1,x_2)=a_1+a_2, d(x_1,x_3)=a_1+a_3$ and $d(x_2,x_3)=a_2+a_3$ (that is, the $a_i$ are the various Gromov products of the $x_j$).  Take three balls $B_i$ around $x_i$ of radii $a_i$. Such balls mutually intersects, which means they totally intersect. Let $p$ be in their total intersection. Using our choice of $B_i$ and $a_i,$ the point $p$ satisfies the conclusion of the lemma (since injective spaces are geodesic, see Proposition \ref{prop:bicombing}).
\end{proof}

 We now show increasingly strong properties of projections to Morse subsets in an injective metric space. First, we show that given a point $x$, a Morse subset $Y$ in an injective metric space, and a point $y$ on $Y$ there exists a geodesic from $x$ to $y$ passing close to the projection of $x$; this is what the proof shows even though later on we only need the inequality stated in the following lemma.
 
 %First, we give a bound on the size of the projection of a point.

% \begin{lemma}
% Let $x$ be a point in an injective metric space $X$ and let $Y$ be an $M$-Morse subset. Then we have $\text{diam}(\pi_Y(x)) \leq 2M(1,0)+2.$
% \end{lemma}

% \begin{proof}
% Let $y,z \in \pi_\alpha(x).$ As in the previous lemma, we can write $d(x,y)=a+b,$ $d(x,z)=a+c$ and $d(y,z)=b+c$ where $a,b,c$ are three non-negative integers. Since $y,z \in \pi_\alpha(x)$, we have $|d(x,y)-d(x,z)|\leq 1$, which implies $|b-c|\leq 1$. By the previous lemma, there exists a point $p$ and a geodesic $\beta=[y,p] \cup [p,z]$ where $[y,p], [p,z]$ are also geodesics. 

% Since $Y$ is $M-$Morse, we have with $d(p,Y) \leq M(1,0).$ In view of Lemma \ref{rmk:proj_geod}, we have $b=d(p, y)\leq d(p,Y)+1 \leq M(1,0)+1$. Thus, $d(y,z)=b+c\leq 2b+1 \leq 2M(1,0)+2.$
% \end{proof}

\begin{lemma}
\label{lem:one_geod}
Let $X$ be an injective metric space. Let $Y$ be an  $M$-Morse subset and let $x \in X$. If $y \in \pi_Y(x)$ and $z \in Y$, then
$$d(x,z)\geq d(x,y)+d(y,z) - 2 M(1,0)-2.$$
%If $y \in \pi_y(x)$ and $z \in Y$, then there is a geodesic $\beta$ such that $d(y, \beta) \leq M(1,0)$.
\end{lemma}

\begin{proof}
We show that if $y \in \pi_Y(x)$ and $z \in Y$, then there is a geodesic $\beta$ such that $d(y, \beta) \leq M(1,0)+1$, which suffices.

%We can write $d(x,y)=a+b,$ $d(x,z)=a+c$ and $d(y,z)=b+c$ where $a,b,c$ are three non-negative integers. Thus, 
By Lemma \ref{lem:tripod}, we get a point $p$ and geodesics $[x,p] \cup [p,z], [x,p] \cup [p,y], [z,p] \cup [p,y].$ Since $Y$ is $M-$Morse and $p$ lies on the geodesic $[z,p] \cup [p,y]$, we have $d(p,Y)\leq M(1,0)$. Further, since $[x,p] \cup [p,y]$ is a geodesic, by Lemma \ref{rmk:proj_geod} we have $d(p,y)\leq d(p,Y)+1$, and in turn $d(p,y)\leq M(1,0)+1$.
\end{proof}

The next lemma improves upon the previous one by showing that, in the same setup, \emph{any} geodesic passes close to the projection point.

\begin{lemma}
\label{lem:all_geod}
Let $X$ be an injective metric space. Let $Y$ be an $M$-Morse subset and let $x \in X$. Then there exists a constant $C,$ depending only on $M$, such that for any $y\in \pi_Y(x), z \in Y,$ and any geodesic $\alpha$ from $x$ to $z$, we have $d(y, \alpha) \leq C.$
\end{lemma}

 %, and set $\delta=2M(1,0)+2$, so that $d(x,z)\geq d(x,y)+d(y,z) -\delta$ by Lemma \ref{lem:one_geod}.
 
\begin{proof} 
 Fix the notation of the statement.
 % If $z \in \pi_Y(x),$ then we are done by Lemma \ref{lem:one_geod}. Namely, Lemma \ref{lem:one_geod} assures that $d(x,z) \geq d(x,y)+d(y,z)-2M(1,0)-2$, and since $z \in \pi_Y(x),$ we get that $d(x,y)+1 \geq d(x,Y)+1 \geq d(x,z)$. Combining such inequalities we get
% $$d(x,y)+1 \geq d(x,z) \geq d(x,y)+d(y,z)-2M(1,0)-2,$$ which gives us that $d(y,z) \leq 3+2M(1,0).$ Hence, we may assume that $z \notin \pi_Y(x)$.
  If $d(x,z)<d(x,y)$ then using Lemma \ref{lem:one_geod} we get
 $$d(x,y) \geq d(x,z) \geq d(x,y)+d(y,z)-2M(1,0)-2,$$
 which gives us $d(y,\alpha)\leq d(y,z) \leq 2+2M(1,0)$. Hence, we may assume $d(x,z)\geq d(x,y)$.
 
 Let $p$ be the point along $\alpha$ lying at distance $d(x,y)$ from $x$. Let now $q\in \alpha_{xp}$ (recall that $\alpha_{xp}$ is the subpath of $\alpha$ from $x$ to $p$) be a point minimising the distance from $y$. Let $[y,q]$ be a geodesic connecting $y$ to $q$ given by the bicombing on $X$ from Proposition \ref{prop:bicombing}. The concatenation $\gamma$ of $[y,q]$ and $\alpha_{qp}$ is a $(3,0)$-quasi-geodesic by Lemma \ref{lem:QRT}.

 We now argue that $p$ is a closest point to $z$ on $\gamma$. Note $p$ is closest to $z$ on $\alpha_{qp}$, since the latter is contained in a geodesic to $z$. Moreover, since $[y,q]$ is assumed to be a bicombing geodesic and both $y,q$ live in a ball $B$ about $x$ of radius $d(x,p)$, the geodesic $[y,q]$ is entirely contained in $B$. Hence, no point on $[y,q]$ can be closer than $d(p,z)$ from $z$, as $d(x,z)=d(x,p)+d(p,z)$. Now, since $p$ is closest, we can again apply Lemma \ref{lem:QRT} and get that the concatenation of $[y,q]$ and $\alpha_{qz}$ is a $(9,0)$-quasi-geodesic. Since $Y$ is Morse and $p$ lies on said concatenation, we have $d(p,p')\leq M(9,0)$ for some $p'\in Y$. Using Lemma \ref{lem:one_geod}, we have
$$d(x,p)\geq d(x,p')-d(p',p)\geq d(x,y)+d(y,p')-2M(1,0)-2-M(9,0),$$
so that $d(y,p')\leq M(9,0)+2M(1,0)+2$, and hence $d(y,p)\leq 2 M(9,0)+2M(1,0)+2$, as required.
\end{proof}

% \begin{figure}
%    \centering
%    $$\includegraphics[width=500pt]{pic.jpg}$$
%  \caption{This is for Lemma 2.5} \label{Proof}
% \end{figure}

Finally, we obtain Theorem \ref{thm:main}.

\begin{theorem}
\label{thm:main_body}
For any injective metric space $X$, a subset is Morse if and only if it is strongly contracting. Moreover, the contracting constant only depends on the Morse gauge and vice versa.
\end{theorem}

In light of the above Lemmas, the proof is exactly the same as the proof of Theorem 2.14 in \cite{charneysultan:contracting}. We include it for completeness.

\begin{proof}
The fact that strongly contracting subsets are Morse quantitatively is \cite[Theorem 1.4]{ACGH}, so we will only prove the converse. Let $Y$ be an $M$-Morse subset, $B$ be a ball centered at $x$ of radius $r$ with $B \cap Y=\emptyset.$ Let $y$ in $B$ and let $x' \in \pi_Y(x), y' \in \pi_Y(y)$. Define $A:=d(x,x')$ and notice that $A \geq r$. Choose a bicombing geodesic $[y,x']$. Lemma \ref{lem:all_geod} ensures the existence of a point $z \in [y,x']$ with $d(z,y') \leq \delta,$ where $\delta$ depends only on $M.$ Now, since $[y,x']$ is a bicombing geodesic, any point $z' \in [y, x']$ satisfies $d(z',x)\leq \max\{d(x,x'),d(x,y)\}\leq A$.
%$$d(z',x) \leq (1-t)d(x,x')+td(x,y) \leq (1-t)A+tr \leq (1-t)A+tA=A.$$
In particular, we have $d(z, x) \leq A.$ Hence, we have 
$$d(x,y') \leq d(x,z)+d(z,y') \leq A+\delta.$$
Now, applying Lemma \ref{lem:all_geod} to the triangle $x,x',y'$ gives us a point $w \in [x,y']$ with $d(x',w) \leq \delta$. Hence, 
$$d(x,y')=d(x,w)+d(w,y') \geq (A-\delta)+(d(x',y')-\delta)=A+d(x',y')-2\delta.$$
Combining the previous two inequalities gives

$$A+d(x',y')-2\delta \leq d(x,y') \leq A+\delta,$$ and hence $d(x',y') \leq 3 \delta.$ Hence, we have $\text{diam}(\pi_Y(B)) \leq 6 \delta$ completing the proof.
\end{proof}

In fact, the arguments above prove the following more general fact:

\begin{corollary}\label{cor:9_0}
Let $Y$ be a subset of an injective metric space. The following statements are all equivalent:

\begin{enumerate}
    \item $Y$ is Morse.
    \item $Y$ is strongly contracting.
    \item There exists an integer $K$ such that every $(9,0)$-quasi-geodesic with end points on $Y$ is contained in the $K$-neighborhood of $Y.$
    
\end{enumerate}

\end{corollary}

We remark that the bound in item 3 of the above corollary is sharper than the known bound in CAT(0) spaces (see Theorem 2.14 in \cite{charneysultan:contracting} and Theorem 3.10 \cite{QRT19}). Namely, for CAT(0) spaces, to ensure that a subset $Y$ is strongly contracting, one needs to require that $(32,0)$-quasi-geodesics with end points on $Y$ remain close to $Y$ while the above corollary only requires $(9,0)$-quasi-geodesics to stay close.
\section{Consequences}\label{sec:MLTG}

\subsection{Morse local-to-global} Roughly speaking, a space is said to be \emph{Morse local-to-global} \cite{Morse-local-to-global} if local Morse quasi-geodesics are global Morse quasi-geodesics.  The goal of this section is to prove that geodesic metric spaces where the notions of Morse and strongly contracting are equivalent are Morse local-to-global. In fact, we shall prove the more general fact that in any geodesic metric space, local strongly contracting quasi-geodesics are strongly contracting, and the conclusion will then follow for spaces where the notions of Morse and strongly contracting agree. We start by introducing the following notion.

\begin{definition}
A map $\gamma \colon I \to X$ is a \emph{$(L;D;k,c)$-local-strongly-contracting quasi-geodesic} if for any $[s,t]\subseteq I$ we have \[|s-t| \leq L \implies \gamma\vert_{[s,t]} \text{ is a } D\text{-strongly contracting } (k,c)\text{-quasi-geodesic}.\]
\end{definition}

The main proposition we shall prove in this section is the following.

\begin{proposition}\label{prop:contraction_local_to_global} Given $D, k, c$, there exists $L=L(D,k,c)$ such that the following holds. Every $(L;D;k,c)$-local-strongly contracting quasi-geodesic in a geodesic metric space is a $D$-strongly contracting $(k',c')$-quasi-geodesic with $k',c'$ depending only on $D,k$ and $c.$
\end{proposition}

Notice that the contraction constant $D$ in the hypotheses and conclusion is the same.

%In \cite{EikeZalloum}, the following lemma is stated under the assumptions that $\gamma$ is a geodesic and $X$ is proper, but neither is used in the proof.

%\begin{lemma}[{\cite[Theorem 1.1 plus Lemma 2.15]{EikeZalloum}}]\label{lem:bounded_jumps}
%Let $\alpha:[a,b] \rightarrow X$ be a $D$-strongly contracting quasi-geodesic and let $c \in [a,b].$ Denote $\alpha_{a,c}, \alpha_{c,b}$ the sub-segments of $\alpha$ connecting the respective points. For any $x \in X,$ if $\pi_\alpha(x) \cap \alpha_{a,c} \neq \emptyset$, then every $p \in \pi_{\alpha_{c,b}}(x)$ satisfies $d(p,\alpha(c))<D'$, for a constant $D'$ depending only on $D.$\marginpar{comment on the "+1" in closest point projections?}
%\end{lemma}

The following is similar to \cite[Lemma 2.15]{EikeZalloum}.

\begin{lemma}(Bounded jumps)\label{lem:bounded_jumps}
    For all $k,c,D$ there exists $D'$ with the following property. Let $\gamma$ be a $D$-strongly-contracting $(k,c)$-quasi-geodesic in a geodesic metric space $X$ and suppose that $\gamma$ consists of the concatenation $\gamma_1*\gamma_2$ of two $D$-strongly contracting $(k,c)$-quasi-geodesics, and call $p$ the concatenation point. Let $x\in X$ be such that $d(\pi_{\gamma_1}(x),p)\geq D'$. Then $d(\pi_{\gamma_2}(x),p)\leq D'$.
\end{lemma}

\begin{proof}
    Suppose by contradiction that we have points $x_i\in \pi_{\gamma_i}(x)$ with $d(x_i,p)\geq D'$ for some $D'=D'(K,C, D)$ to be determined (notice that in this formulation the statement is symmetric in $\gamma_1,\gamma_2$). Up to swapping indices, we can assume $d(x_1,x)\leq d(x_2,x)=:d$. We have $d(x,\gamma_2)\geq d-1$ by definition of projections, so that some point $x'_1$ within distance 1 of $x_1$ on a geodesic $[x,x_1]$ lies in a ball $B$ around $x$ disjoint from $\gamma_2$. Therefore $d(q, x_2)\leq D$ for any $q\in \pi_{\gamma_2}(x'_1)$. We now argue that $d(q,p)$ is uniformly bounded. This is because any geodesic $[x'_1,q]$ stays within distance bounded by some $B=B(D)$ from the subpath of $\gamma$ from $x_1$ to $q$, since said subpath is Morse with gauge depending on $D$ only. In particular, $[x'_1,q]$ passes $B$-close to $p$, and hence we would have $d(x'_1,p)< d(x'_1,q)-1$ if we have $d(p,q)>B+1$; therefore $d(p,q)\leq B$ as required. Hence, $d(x_2,p)\leq D+B$, a contradiction for $D'>D+B$.
\end{proof}

The proof of Proposition \ref{prop:contraction_local_to_global} will be broken into the following two lemmas. The first one establishes that local strongly contracting quasi-geodesics are strongly contracting \emph{sets}, and the second proves that they are indeed quasi-geodesics.

\begin{lemma} Let $D,k,c$ be given. There exists an $L=L(D,k,c)$ such that each $(L;D;k;c)$-local-strongly contracting path $\gamma:[a,b] \rightarrow X$, where $X$ is a geodesic metric space, is a $D$-strongly contracting set.
\end{lemma}

\begin{proof} Choose $L$ large enough to be determined later. Let $B$ be a ball disjoint from $\gamma$, and we claim that the projection of $B$ to $\gamma$ is entirely contained in a subsegment $\gamma|_I \subseteq \gamma$ with $|I| \leq L$. Once this claim is established we will be finished as $\gamma|_I$ is $D$-strongly contracting. 

Suppose the claim does not hold, that is to say, $B$ contains two points $x,y$ whose projections $\pi_\gamma(x), \pi_\gamma(y)$ contain points $x',y'$ such that the $\gamma$-subsegment connecting $x',y'$, denoted $\gamma|_J$, satisfies $|J|>L.$ Subdivide $\gamma|_J$ into subsegments $\gamma|_{J_1}, \cdots, \gamma|_{J_n}$ with
$$(2D'+D)k+kc+D<|J_i| \leq \frac{L}{2}$$
for each $i \in \{1,2,\cdots n\}$, where $D'$ is as in Lemma \ref{lem:bounded_jumps} (this can be done provided $L> 4((2D'+D)k+kc+D)$). In particular, $x'$ is the initial point of $\gamma|_{J_1}$ and $y'$ is the terminal point of $\gamma|_{J_n}$. Since each $|J_i| \leq \frac{L}{2}$ for $i\geq 1,$ the subsegments $\gamma|_{J_i \cup J_{i+1}}$ are all $(k,c)$-quasi-geodesic which are $D$-strongly contracting. Notice that $y'$, which is the terminal point of $\gamma|_{J_n},$ lives in $\pi_{\gamma|_{J_{n}}}(y)$. Hence, applying Lemma \ref{lem:bounded_jumps} to $\gamma|_{J_{n-1} \cup J_n}=\gamma|_{J_{n-1}} \ast \gamma|_{J_n}$ we see that each projection of $y$ to $\gamma|_{J_{n-1}}$ is within $D'$ of the point $\gamma|_{J_{n-1}} \cap \gamma|_{J_{n}}$, hence further than $D'$ from the other endpoint of $\gamma|_{J_{n-1}}$ since the distance between the endpoints of $\gamma|_{J_{n-1}}$ is at least $|J_{n-1}|/k-c>2D'+D\geq 2D'$. Repeating this process $(n-1)$-times shows that the projection of $y$ to $\gamma|_{J_1}$ is within $D'$ of the endpoint $\gamma|_{J_1}\cap \gamma_{J_2}$ of $\gamma|_{J_1}$. That is, the $D'$-neighborhoods of the two endpoints of $\gamma|_{J_1}$ contain points from $\pi_{\gamma|_{J_1}}(x)$ and $\pi_{\gamma|_{J_1}}(y)$ respectively. Since the distance between the endpoints of $\gamma|_{J_1}$ is larger than $2D'+D$ (as shown above for $\gamma|_{J_{n-1}}$), we see that $B$, which is disjoint from $\gamma|_{J_1}$, has projection of size larger than $D$, contradicting that $\gamma|_{J_1}$ is $D$-strongly contracting. This concludes the proof.
\end{proof}

\begin{lemma}
Continuing with the assumptions and notations from the previous lemma, the path $\gamma$ is a $(k',c')$-quasi-geodesic where $k',c'$ depend only on $D,k,c.$
\end{lemma}

\begin{proof} Let $L$ be the constant from the previous lemma. Since our ultimate goal is to prove Proposition \ref{prop:contraction_local_to_global}, we may enlarge $L$, as long as its value depends only on $D,k,c.$ Let $\beta$ be a geodesic connecting the endpoints of $\gamma$ and let $\gamma|_J$ be any sub-path of $\gamma$ with $|J| = L.$ Using our assumptions, the path $\gamma|_J$ is a $(k,c)$-quasi-geodesic which is $D$-strongly contracting. Arguing exactly as in the previous lemma, if $x,y$ are the endpoints of $\gamma,$ then $\pi_{\gamma|_J}(x), \pi_{\gamma|_J}(y) $ are within $D'=D'(D,k,c)$ of the two (distinct) endpoints of $\gamma|_J$. In particular, $d(\pi_{\gamma|_J}(x), \pi_{\gamma|_J}(y)) \geq L-2D'.$ Since $\gamma|_J$ is $D$-strongly contracting, if $L$ is large enough (depending only on $D,k,c)$, the geodesic $\beta$ must pass $5D$-close to $\gamma|_J$ (see e.g. the proof of \cite[Lemma 4.4]{charneysultan:contracting}, see also \cite{contractingrandom,ACGH, EikeZalloum}). That is, the geodesic $\beta$ must pass within $5D$ of every subsegment $\gamma_J$ with $|J|=L$. This ensures that $\gamma$ is a $(k',c')$-quasi-geodesic with $k',c'$ depending only on $L,5D$ both of which depend only on $D,k,c,$ concluding the proof.
\end{proof}

In particular, the previous two lemmas provide a proof of Proposition \ref{prop:contraction_local_to_global}. For the convenience of the reader, we now recall the definition of the Morse local-to-global property:

\begin{definition}[{\cite[Definition 2.12]{Morse-local-to-global}}]\label{defn:Morsel_local_to_global}
Let $M \colon [1,\infty) \times [0,\infty) \to [0,\infty)$ and $B \geq 0$.
If $\gamma$ is an $M$-Morse $(k,c)$-quasi-geodesic, we say $\gamma$ is a $(M;k,c)$-Morse quasi-geodesic. A map $\gamma \colon I \to X$ is a \emph{$(B;M;k,c)$-local Morse quasi-geodesic} if for any $[s,t]\subseteq I$ we have \[|s-t| \leq B \implies \gamma\vert_{[s,t]} \text{ is a } (M;k,c)\text{-Morse quasi-geodesic}.\]

A metric space $X$ has the \emph{Morse local-to-global property} if for every Morse gauge $M$ and constants $k\geq 1$, $c \geq 0$, there exists a local scale $B \geq 0$, Morse gauge $M'$, and constants $k' \geq 1$, $c'\geq 0$ such that every $(B;M;k,c)$-local Morse quasi-geodesic is a global $(M';k',c')$-Morse quasi-geodesic. 
\end{definition}

The following proposition is immediate by combining Theorem \ref{thm:main_body} and Proposition \ref{prop:contraction_local_to_global}.

\begin{proposition}\label{thm:MLTG_body}
Every metric space where Morse quasi-geodesics are strongly contracting quantitatively has the Morse local-to-global property. In particular, injective metric spaces and Helly groups have the Morse local-to-global property.
\end{proposition}

\subsection{Acylindrical hyperbolicity} \label{subsec:a.h}

\begin{corollary}\label{thm:ah_body}
If $G$ is a group acting properly and coboundedly on an injective metric space $X$ whose Morse boundary $\partial_M X$ contains at least 3 points, then $G$ is acylindrically hyperbolic. In particular, Helly groups whose Morse boundary contain at least 3 points are acylindrically hyperbolic.
\end{corollary}

\begin{proof}
Since $X$, whence $G$, has the Morse local-to-global property (Theorem \ref{thm:injectiveareMLTG}), by \cite[Proposition 4.8]{Morse-local-to-global}, $G$ contains a Morse element. In view of Theorem \ref{thm:main} said element has a strongly contracting orbit for the action on $X$. By \cite[Theorem H]{Bestvina2014}, $G$ acts on a hyperbolic space with a loxodromic WPD, so that $G$ is virtually cyclic or acylindrically hyperbolic. Since $G$ has 3 points in the Morse boundary it is not virtually cyclic, concluding the proof.
\end{proof}

\section{Strongly contracting geodesics persist in injective hulls}
\label{sec:contr_survives}

Given a metric space $X$, consider the set $\Delta:=\{f \in \mathbb{R}^X | f(x)+f(y) \geq d(x,y) \,\, \forall x,y \in X\}.$ An element $f \in \Delta$ is said to be of a \emph{metric form}. Define the set $\Delta^1:=\Delta \cap \text{Lip}^{1}(X, \mathbb{R})$ where $\text{Lip}^{1}(X, \mathbb{R})$ denotes the space of all $1$-Lipshitz maps $f:X \rightarrow \mathbb{R}$.  We equip the set $\Delta^1$ with the distance $d_\infty(f,g):=\text{sup}\{|f(x)-g(x)| \text{ where } \,\, x \in X\}$.

The space $\Delta$ has a natural poset structure where $f \leq g$ if and only if $f(x) \leq g(x)$ for all $x \in X.$ The \emph{injective hull} of a metric space $X$, denoted $E(X)$, is defined to be the collection of all minimal elements in the aforementioned poset on $\Delta$:
$$E(X):=\{f \in \Delta| \,\,\text{if}\,\, g \in \Delta, \text{we have}\,\, g \leq f \implies g=f\}.$$
 
As a metric space, the injective hull is the set $E(X)$ equipped with the $d_\infty$-distance.
We now summarize some properties of $E(X).$ For more details, see Section 3 of \cite{LANG2013}.

\begin{lemma}[{\cite[Section 3]{LANG2013}}]\label{lem:summary} For any metric space $X$, the injective hull $E(X)$ is a geodesic metric space with respect to $d_\infty$-distance. Furthermore, we have the following.

\begin{enumerate}
    \item There is an isometric embedding $e:(X,d) \rightarrow (E(X),d_\infty)$ given by $x \mapsto d(x,-).$
    \item For any $f \in \Delta^1,$ we have $f(x)=d_\infty(e(x),f)$ for all $x \in X.$
    
    \item  For $f \in \mathbb{R}^X$, we have $f \in E(X)$ if and only if  $f \in \Delta$ and for any $\epsilon >0$ and any $x \in X,$ there is some $y \in X$ with $f(x)+f(y) \leq d(x,y)+\epsilon.$ 
    
    \item $E(X) \subseteq \Delta^1.$

    \end{enumerate}

\end{lemma}

Recall that Gromov's 4-point condition for a metric space $X$ requires the existence of a constant $\delta$ such that for all $x,y,w,z\in X$ we have 
$$d(x,y)+d(w,z)\leq \max\{d(x,z)+d(w,y),d(x,w)+d(y,z)\}+\delta.$$
For geodesic spaces, this condition is equivalent to hyperbolicity.

In the proofs below constants could be kept track of explicitly, but that would result in complicated expressions, so we decided instead to use the notation $x\approx y$ to denote that the quantities $x$ and $y$ coincide up to an additive error depending on the constant $C$ featured in the statements. The notation $x\lesssim y$ has a similar meaning. We will often abuse notation and identify a point $x \in X$ with its image $e(x)$.

The following lemma says that a strongly contracting and hyperbolic subset stays quasiconvex (with respect to geodesics) in the injective hull.

\begin{lemma}
\label{lem:contr_inj}
For all $C \geq 0$ there exists $D \geq 0$ with the following property. Let $X$ be a metric space and let $A\subseteq X$ be a $C$-strongly contracting subset satisfying Gromov's 4-point condition with constant $C$. Then any geodesic in $E(X)$ connecting points of $e(A)$ is contained in the $D$-neighborhood of $e(A)$.
\end{lemma}

\begin{proof}
To save notation, we will identify $X$ with $e(X)$. Let $f\in E(X)$ lies on a geodesic connecting $a$ to $b$ for some $a,b\in A$. Since $f$ lives on a geodesic connecting $a,b$, we have $d_\infty(f,a)+d_\infty(f,b)=d_\infty(a,b)=d(a,b).$ On the other hand, by items 2 and 4 of Lemma \ref{lem:summary}, we have $f(a)=d_\infty(f,a)$ and $f(b)=d_\infty(f,b).$ Hence, we have $f(a)+f(b)=d(a,b)$. Since geodesics in $X$ connecting points of $A$ stay close to $A$, there exists $c\in A$ with $d(c,a)\approx f(a)$ and $d(c,b)\approx f(b)$. We want to estimate $f(c)=d_\infty(f,e(c))$. By item 3 of Lemma \ref{lem:summary}, since $f \in E(X),$ there exists $x\in X$ with $f(c)\approx d(x,c)-f(x)$, and by strong contraction we have $d(x,c)\approx d(x,x')+d(x',c)$ for some $x' \in \pi_A(x)$. Putting these together we have:

$$d(a,b)+f(c) \approx d(a,b) +d(x',c) + d(x,x') -f(x).$$

We now have two cases from Gromov's 4-point condition, we do the case $d(a,b) +d(x',c)\leq d(a,x')+d(b,c)+C$, the other case being similar. Continuing the inequality above, we get

$$\lesssim d(a,x')+d(b,c) +d(x,x') -f(x) \approx d(a,x)-f(x) +d(b,c),$$
where again we used strong contraction. Since $f$ is a metric form, we continue with 

$$ \lesssim f(a)+d(b,c)\approx d(a,c)+d(c,b)\approx d(a,b).$$

Putting everything together we get $d(a,b)+f(c)\lesssim d(a,b)$, so that $f(c)\approx 0$, as required.
\end{proof}

The following lemma gives the "reverse triangle inequality" which would follow from knowing that geodesics to (the image of) a strongly contracting and hyperbolic subspace of $X$ pass close to projection points. We will actually use this lemma to prove this fact about geodesics later.

\begin{lemma}
\label{lem:weak_contr_inj}
For all $C \geq 0$ there exists $D \geq 0$ with the following property. Let $A\subseteq X$ be $C$-strongly contracting satisfying Gromov's 4-point condition with constant $C$. Given any $x\in E(X)$, $a\in \pi_{e(A)}(x)$ and $b\in e(A)$ we have
$$d_\infty(x,b)\geq d_\infty(x,a)+d_\infty(a,b)-D.$$
\end{lemma}

\begin{proof}
Using Lemma \ref{lem:contr_inj}, the proof proceeds exactly as in Lemma \ref{lem:one_geod} (where the Morse property was only applied to geodesics). 
\end{proof}

The following lemma is not about injective hulls, but rather it provides a way of thinking about strong contraction of a hyperbolic subspace $Y \subset X$ in terms of Gromov's 4-point condition. Namely, for any $x,y\in X$, sets of the form $Y\cup\{x,y\}$ still satisfy Gromov's 4-point condition.

\begin{lemma}\label{lem:adding_points} For all $C\geq 0$ there exists $\delta\geq 0$ with the following property. Let $X$ be a geodesic metric space and let $Y$ be a $C$-strongly contracting subset satisfying Gromov's 4-point condition with constant $C$. Then for any $x,y \in X,$ the set $Y'=Y \cup \{x,y\}$ satisfies Gromov's 4-point condition with constant $\delta$.
\end{lemma}

\begin{proof}
The lemma easily follows from the following well-known consequence of strong contraction (see e.g. the proof of \cite[Lemma 4.4]{charneysultan:contracting}, see also \cite{contractingrandom,ACGH, EikeZalloum}): There exists $C'=C'(C)$ such that for all $x,y\in X$ and $x'\in \pi_A(x),y'\in\pi_A(y)$ with $d(x',y')\geq C'$ we have
$$d(x,y)\geq d(x,x') + d(x',y')+d(y',y)-C'.$$
\end{proof}

The following lemma says that in $E(X)$ geodesics to (the image of) a strongly contracting and hyperbolic subspace of $X$ pass close to projection points.

\begin{lemma}
\label{lem:two_cases}
For all $C \geq 0$ there exists $D \geq 0$ with the following property. Let $X$ be a geodesic metric space and let $A\subseteq X$ be a $C$-strongly contracting subset satisfying Gromov's 4-point condition with constant $C$. Given any $f\in E(X)$, $a\in \pi_{e(A)}(f)$ and $b\in e(A)$, any geodesic in $E(X)$ from $f$ to $b$ passes within $D$ of $e(a)$.
\end{lemma}

\begin{proof}
Consider a geodesic as in the statement and let $g$ be a point along that geodesic with $d_\infty(f,g)=d_\infty(f,e(a))$, so that $d_\infty(g,e(b))\approx d_\infty(e(a),e(b))$ by Lemma \ref{lem:weak_contr_inj}. These can be rewritten as $d_\infty(f,g)=f(a)$ and $g(b)\approx d(a,b)$, using item 2 of Lemma \ref{lem:summary}. Our aim is to bound $g(a)$. Using items 3,4 of Lemma \ref{lem:summary}, there exist $x,y\in X$ such that 
$$g(a)+f(b)\approx d(x,a)-g(x)+ d(y,b)-f(y).$$

Since $A$ is strongly contracting and $x,y \in X,$ the set $A'=A \cup \{x,y\}$ satisfies the Gromov's 4-point condition by Lemma \ref{lem:adding_points}. Since $X$ isometrically embeds in $E(X),$ the set $e(A')$ also satsfies Gromov's 4-point condition.

Applying Gromov's 4-point condition on the points $a,b,x,y$ in the above coarse equality, we either get

\begin{align*}
    g(a)+f(b)&\approx d(x,a)+d(y,b)-g(x)-f(y)\\
    &\lesssim d(x,b)+d(y,a)-g(x)-f(y)\\
    &\leq g(b)+f(a)\\
    &\approx d(a,b)+f(a)\\
    &\approx f(b),
\end{align*}

or 

\begin{align*}
    g(a)+f(b)&\approx d(x,a)+d(y,b)-g(x)-f(y)\\
    &\lesssim d(a,b)+d(x,y)-g(x)-f(y)\\
    &\leq g(y)-f(y)+d(a,b)\\
    &\leq f(a)+d(a,b)\\
    &\approx f(b),
\end{align*}
where the fourth inequality holds since $g(y)\leq f(y)+d_\infty(f,g)=f(y)+f(a)$ by the triangle inequality and the defining property of $g$.
\end{proof}

\begin{theorem}\label{thm:contracting_survives}
For all $C$ there exists $D$ with the following property. Let $X$ be a geodesic metric space and let $A\subseteq X$ be a $C$-strongly contracting subset satisfying Gromov's 4-point condition with constant $C$. Then $e(A)$ is $D$-strongly contracting in $E(X)$.
\end{theorem}

\begin{proof}
Same as the proof of Theorem \ref{thm:main_body}, using Lemma \ref{lem:two_cases} instead of Lemma \ref{lem:all_geod}.
\end{proof}

In particular, Theorem \ref{thm:contracting_survives} above recovers the following statement due to \cite{LANG2013}.
\begin{corollary}
\label{cor:Lang}
If $X$ is a hyperbolic metric space with respect to the Gromov's 4 point condition, then $X$ is dense in $E(X)$.
\end{corollary}

\begin{proof}
Let $f \in E(X)$ and let $e(a) \in \pi_{e(X)}(f).$ By Lemma \ref{lem:summary}, we have $f(a)=d_\infty(e(a), f)$ and there exists a point $b \in X$ with $f(a)+f(b) \lesssim d(a,b).$ Since $e(X)$ is $D$-strongly contracting, we have $f(b)=d_\infty(f,e(b)) \approx d_\infty(f,e(a))+d_\infty(e(b),e(a))=f(a)+d(a,b)$. Finally, since $f(a)+f(b) \lesssim d(a,b) \approx f(b)-f(a)$ we get that $f(a) \lesssim -f(a)$ or $f(a) \approx 0.$
\end{proof}

\bibliography{bio}{}
\bibliographystyle{alpha}
\end{document}